\documentclass[preprint,12pt]{article}

\usepackage{amsthm}
\usepackage{amssymb}
\usepackage{amsfonts}
\usepackage{tabulary}
\usepackage{amsmath}
\usepackage{textcomp}
\usepackage{multirow}
\usepackage{hyperref}

\newtheorem{theorem}{Theorem}[section]

\newtheorem{corollary}[theorem]{Corollary}

\newtheorem{example}[theorem]{Example}
\newtheorem{lemma}[theorem]{Lemma}
\newtheorem{proposition}[theorem]{Proposition}
\newtheorem{remark}[theorem]{Remark}

\begin{document}

\title{Totally umbilic null hypersurfaces in generalized Robertson-Walker spaces\footnote{This paper was supported in part by Grupo Junta de Andaluc\'ia  FQM-324.}}
\author{M. Guti\'{e}rrez\footnote{The first author was supported in part by MEYC-FEDER Grant MTM2013-41768-P.} \\ Dep. \'{A}lgebra, Geometr\'{\i}a y 
Topolog\'{\i}a. Universidad de M\'{a}laga.
\and B. Olea \\ Dep Matem\'atica Aplicada. Universidad de M\'alaga.}
\maketitle

\begin{abstract}
We show that there is a correspondence between totally umbilic null
hypersurfaces in generalized Robertson-Walker spaces and twisted
decompositions of the fibre. This allows us to prove that nullcones are the
unique totally umbilic null hypersurfaces in the closed Friedmann Cosmological
model. We also apply this kind of ideas to static spaces, in particular to
Reissner-Nordstr\"{o}m and Schwarzschild exterior spacetimes.
\end{abstract}

\textbf{keyword:} null hypersurface, totally
umbilic hypersurfaces, twisted products, generalized Robertson-Walker space,
static space.


%\author{Manuel Guti\'errez\footnote{The first author was supported in part by MEYC-FEDER Grant MTM2013-41768-P.}}
%\ead{mgl@agt.cie.uma.es}
%\address{Departamento de \'Algebra, Geometr\'ia y Topolog\'ia.\\ Universidad de M\'alaga. Spain.}

%\author{Benjam\'in Olea}
%\ead{benji@uma.es}
%\address{Departamento de Matem\'aticas.\\Instituto E.S. Mar de Albor\'an. M\'alaga. Spain.}

\section{Introduction}

A hypersurface in a Lorentzian manifold is null if the induced metric tensor
is degenerate on it. There is an increasing interest on these hypersurfaces
both from a physical and a geometrical point of view. Black hole horizons are
one of the most remarkable examples \cite{Ast99,GouJar06}. On the other hand,
nullcones play a central role in causality theory and its regularity is of key
importance in the propagation properties of linear and nonlinear waves,
\cite{KlaRod2008}. Specific techniques are needed to study these hypersurfaces
since it is not possible to define an orthogonal projection over them, so
neither the induced connection nor the second fundamental form can be defined
in the usual way.

In this paper, we focus on geometrical properties of null hypersurfaces in a
generalized Robertson-Walker space. The main result is Theorem
\ref{umbilicGRW}, where we show that a totally umbilic null hypersurface gives
rise to a local decomposition of the fibre as a twisted product and viceversa,
providing a deep insight of twisted decompositions in Lorentzian geometry. For
example, not all manifold admits such a decomposition \cite{GutOle12,PonRec93}%
, hence there are spacetimes that do not admit totally umbilic null
hypersurfaces. Moreover, given a totally umbilic null hypersurface, we can
construct a dual one, considering the same induced twisted decomposition in
the fibre but reversing the base. In particular, totally umbilic null
hypersurfaces through each point in generalized Robertson-Walker spaces appear
in pairs. Surprisingly, there are cases in which the dual construction is not
trivial, see Example \ref{Ej2}.

One of the most important examples of null hypersurfaces are nullcones, so we
dedicate Section \ref{seccionconos} to study them. We show that nullcones in
Robertson-Walker spaces are totally umbilic. Nullcones with this property in
generalized Robertson-Walker spaces have special importance because if there
is a null geodesic starting at the vertex of a totally umbilic nullcone and it
has a conjugate point along it, its multiplicity is maximum, Proposition
\ref{Jacobi}. This is potentially interesting in Cosmology and extends a
previous result for Robertson-Walker spaces, \cite{FloSan00}. We also give
necessary and sufficient conditions for a null hypersurface to be an open set
of a nullcone, which, jointly with Theorem \ref{umbilicGRW}, allows us to show
in Section \ref{seccionumbilica} that any totally umbilic null hypersurface in
a Robertson-Walker space $I\times_{f}\mathbb{S}^{n-1}$ with $\int_{I}\frac
{1}{f}>\pi$ is an open set of a nullcone, Theorem \ref{teoremaintegralf}. As a
corollary, we can apply this theorem to the closed Friedmann Cosmological model.

Finally, in Section \ref{seccionestatico}, we adapt Theorem \ref{umbilicGRW}
to the case of standard static spaces and we apply it to an important family
of static spacetimes. We prove a uniqueness result for dual pairs of totally
umbilic null hypersurfaces in Reissner-Nordstr\"{o}m and De
Sitter-Schwarzschild spacetimes.

\section{Preliminaries}

Take $I$ an open interval of $\mathbb{R}$, $f\in C^{\infty}(I)$ a positive
function and $(F,g_{F})$ a Riemannian manifold. The manifold $I\times F$
furnished with the Lorentzian metric $g=-dt^{2}+f(t)^{2}g_{F}$ is called a
generalized Robertson-Walker (GRW) space and is denoted by $I\times_{f}F$.
When $F$ has constant curvature, it is called a Robertson-Walker (RW) space
and if $f\equiv1$, then it is simply denoted by $I\times F$. The vector field
$\zeta=f\partial_{t}$ is timelike, closed and conformal. It locally
characterizes these spaces and, under certain conditions, it is the unique
vector field with these properties, \cite{GutOle09}. If we consider the same
construction as above, but with $dt^{2}$ instead of $-dt^{2}$, the resulting
product metric is called warped product and if we consider a positive function
$f\in C^{\infty}(I\times F)$ instead of $f\in C^{\infty}(I)$, then it is
called a twisted product.

Given $L$ a null hypersurface in a GRW space, we take the unique null vector
field $\xi\in\mathfrak{X}(L)$ such that $g(\zeta,\xi)=1$ and we call
$\mathcal{S}$ the distribution in $L$ given by $\zeta^{\perp}\cap TL$.

The vector field $\xi$ is geodesic and $\nabla_{X}\xi$ is a section of $TL$
for all $X\in\mathcal{S}$. In fact, $\nabla_{X}\xi\in\mathcal{S}$, since
$g(\nabla_{X}\xi,\zeta)=-g(\xi,\nabla_{X}\zeta)=0$. The null second fundamental
form of $L$ is defined by $B(X,Y)=-g(\nabla_{X}\xi,Y)$ for all $X,Y\in
\mathcal{S}$. It is said that it is totally geodesic if $B\equiv0$ and totally
umbilic if $B=\rho g$ for certain $\rho\in C^{\infty}(L)$. The trace of $B$ is
the null mean curvature of $L$, explicitly given by
\[
H_{p}=\sum_{i=3}^{n}B(e_{i},e_{i}),
\]
being $\{e_{3},\ldots,e_{n}\}$ an orthonormal basis of $\mathcal{S}_{p}$.

If $L$ is totally umbilic, then the null sectional curvature respect to $\xi$
of a null plane $\Pi=span(X,\xi)$, where $X\in\mathcal{S}$ is unitary,
can be expressed as
\begin{equation}
\mathcal{K}_{\xi}(\Pi)=\xi(\rho)-\rho^{2}. \label{Eq1}%
\end{equation}

If it is totally geodesic, we have $\mathcal{K}_{\xi}(\Pi)=0$ for any null
tangent plane $\Pi$ to $L$.

If $\theta$ is an open subset of $F$ and $h:\theta\rightarrow I$ is a
function, then the graph of $h$ is a null hypersurface of $I\times_{f}F$ if
and only if
\begin{equation}
\left\vert \nabla^{F}h\right\vert _{F}=f\circ h. \label{eqhyperluzGRW}%
\end{equation}

Locally, any null hypersurface $L$ can be expressed in this way. If we call
$\pi:I\times_{f}F\rightarrow F$ and $T:I\times_{f}F\rightarrow I$ the
canonical projections, then $\pi:L\rightarrow F$ is a local diffeomorphism and
thus, locally, $L$ coincides with the graph of the function given by
$h=T\circ\pi^{-1}:\theta\rightarrow I$, where $\theta\subset F$.
Moreover, given $v\in TF$ we have
\[
g_{F}(\nabla_{\nabla^{F}h}^{F}\nabla^{F}h,v)=\frac{1}{2}v\left(  (f\circ
h)^{2}\right)  =(f\circ h)(\,f^{\prime}\circ h)g_{F}(v,\nabla^{F}h),
\]
therefore
\begin{equation}
\nabla_{\nabla^{F}h}^{F}\nabla^{F}h=(f\circ h)(\,f^{\prime}\circ h)\nabla
^{F}h. \label{remarkh}%
\end{equation}

\section{Nullcones in generalized Robertson-Walker spaces}

\label{seccionconos}

If $M$ is a Lorentzian manifold and $\Theta$ a normal neighborhood of a point
$p\in M$, then we call $\widehat{\Theta}=exp_{p}^{-1}(\Theta)$ and
$\widehat{P}$ the position vector field in $T_{p}M$. The local position vector
field at $p$ is defined as $P_{\exp_{p}(v)}=(\exp_{p})_{\ast_{v}}(\widehat
{P}_{v})$ for all $v\in\widehat{\Theta}$ and the local future and past
nullcones at $p$ as
\[
C_{p}^{+}=exp_{p}\left(  \widehat{C}_{p}^{+}\cap\widehat{\Theta}\right)
,\ C_{p}^{-}=exp_{p}\left(  \widehat{C}_{p}^{-}\cap\widehat{\Theta}\right)  ,
\]
being $\widehat{C}_{p}^{+}$, $\widehat{C}_{p}^{-}$ the future and past
nullcone respectively in $T_{p}M$. In a GRW space, local nullcones can be
characterized as follows.

\begin{proposition}
\label{conoGRW}Let $I\times_{f}F$ be a GRW space and fix $p_{\ast}=(t_{\ast
},x_{\ast})\in I\times F$. If $\Theta$ is a normal neighborhood of $p_{\ast}$,
then the local nullcones at $p_{\ast}$ are given by
\begin{align*}
C_{p_{\ast}}^{+}  &  =\{(t,x)\in\Theta:\int_{t_{\ast}}^{t}\frac{1}%
{f(r)}dr=d_{F}(x_{\ast},x)\},\\
C_{p_{\ast}}^{-}  &  =\{(t,x)\in\Theta:\int_{t}^{t_{\ast}}\frac{1}%
{f(r)}dr=d_{F}(x_{\ast},x)\},
\end{align*}
being $d_{F}$ the Riemannian distance in $F$. Moreover, the local position
vector field at $p_{\ast}$ is given by
\[
P_{(t,x)}=\int_{t_{\ast}}^{t}\frac{f(r)}{f(t)}dr\,\partial_{t}+\frac
{\int_{t_{\ast}}^{t}\frac{f(r)}{f(t)}dr}{\int_{t_{\ast}}^{t}\frac{f(t)}%
{f(r)}dr}P_{x}^{F},
\]
for all $(t,x)\in C_{p_{\ast}}^{+}\cup C_{p_{\ast}}^{-}$, where $P^{F}$ is the
local position vector field at $x_{\ast}$ in $F$.
\end{proposition}

\begin{proof}
Given $(t,x)\in C_{p_{\ast}}^{+}$, it exists a null geodesic $\gamma
:J\rightarrow C_{p_{\ast}}^{+}$ such that $\gamma(0)=p_{\ast}$, $g(\gamma
^{\prime}(0),\zeta_{p_{\ast}})=-1$ and $\gamma(s^{\ast})=(t,x)$ for certain
$s^{\ast}\in J$. Since $\zeta$ is closed and conformal, $g(\gamma
^{\prime},\zeta)$ is constant and therefore, if $\gamma(s)=(\alpha
(s),\beta(s))$, we have $\alpha^{\prime}(s)f(\alpha(s))=1$. Hence
$\alpha(s)=a^{-1}(s)$, being $a(s)=\int_{t_{\ast}}^{s}f(r)dr$. On the other
hand, $\beta$ is a pregeodesic in $F$ which holds
\[
\beta^{\prime\prime}(s)=-2\frac{d}{ds}\Big(\ln f(\alpha(s))\Big)\beta^{\prime
}(s),
\]
so $\beta$ is given by
\[
\beta(s)=exp_{x_{\ast}}^{F}\left(  b(s)u\right)  ,
\]
where $b(s)=\int_{0}^{s}\frac{1}{f(\alpha(r))^{2}}dr$ and $u\in T_{x_{\ast}}F$
with $g_{F}(u,u)=1$. Therefore,
\[
d_{F}(x_{\ast},x)=b(s^{\ast})=\int_{t_{\ast}}^{t}\frac{1}{f(r)}dr.
\]

Conversely, take $(t,x)\in\Theta$ such that $d_{F}(x_{\ast},x)=\int_{t_{\ast}%
}^{t}\frac{1}{f(r)}dr$. If we call
\[
a(s)=\int_{t_{\ast}}^{t}f(r)dr, \ b(s)=\int_{0}^{s}\frac{1}{f(a^{-1}(r))^{2}}dr,
\]
$s^{\ast}=a(t)$, $\alpha(s)=a^{-1}(s)$ and $\beta(s)=exp_{x_{\ast}}%
^{F}(b(s)u)$ where $u\in T_{x_{\ast}}F$ is such that $exp_{x_{\ast}}%
^{F}\left(  d_{F}(x_{\ast},x)u\right)  =x$, then it is easy to show that
$\gamma(s)=(\alpha(s),\beta(s))$ is a future null geodesic in $M$ with
$\gamma(0)=p_{\ast}$ and $\gamma(s^{\ast})=(t,x)$. Therefore, $(t,x)\in
C_{p_{\ast}}^{+}$.

To compute the local position vector field over $C_{p_{*}}^{+}$ we observe
that for any manifold and any geodesic $\gamma$ with $\gamma(0)=p_{\ast}$, the
local position vector field is given by $P_{\gamma(s)}=s\gamma^{\prime}(s)$.
In our particular case, $P_{(t,x)}=s^{\ast}\gamma^{\prime}(s^{\ast})$. Since
$\alpha(s^{\ast})=t$, we have
\begin{align*}
s^{\ast}  &  =\int_{t_{*}}^{t}f(r)dr,\\
\alpha^{\prime}(s^{\ast})  &  =\frac{1}{f(t)},\\
\beta^{\prime}(s^{\ast})  &  =b^{\prime}(s^{\ast})\left(  exp_{x_{*}}%
^{F}\right)  _{\ast_{b(s^{\ast})u}}\Big(u\Big)=\frac{b^{\prime}(s^{\ast}%
)}{b(s^{\ast})}P_{x}^{F},
\end{align*}
where $P^{F}$ is the local position vector field at $x_{*}$ in $F$. Since
$b(s^{\ast})=\int_{t_{*}}^{t}\frac{1}{f(r)}dr$ and $b^{\prime}(s^{\ast}%
)=\frac{1}{f(t)^{2}}$, we have
\[
P_{(t,x)}=\int_{t_{*}}^{t}\frac{f(r)}{f(t)}dr\,\partial_{t}+\frac{\int_{t_{*}%
}^{t}\frac{f(r)}{f(t)}dr}{\int_{t_{*}}^{t}\frac{f(t)}{f(r)}dr}P_{x}^{F}.
\]

\end{proof}

The following lemma will be helpful to compute the null second fundamental form of
a nullcone in a Robertson-Walker space.

\begin{lemma}
\label{lemaJacobi} Let $(F,g_{F})$ be a semi-Riemannian manifold, fix
$x_{\ast}\in F$ and take $\theta\subset F$ a normal neighborhood of $x_{\ast}%
$. Call $P^{F}\in\mathfrak{X}(\theta)$ the local position vector field at
$x_{\ast}$. If $w\in T_{x}F$, being $x=exp_{x_{\ast}}^{F}(v)\in\theta$, then
\[
g_{F}(\nabla_{w}^{F}P^{F},w)=\frac{1}{2}\frac{d}{ds}g(J,J)|_{s=1},
\]
where $J$ is the unique Jacobi vector field over $exp_{x_{\ast}}^{F}(sv)$ with
$J(0)=0$ and $J(1)=w$.
\end{lemma}

\begin{proposition}
\label{conosumbilicos}Local nullcones in a Robertson-Walker space are totally umbilic.
\end{proposition}

\begin{proof}
Fix $p_{\ast}=(t_{\ast},x_{\ast})\in I\times_{f} F$ and consider $p=(t,x)\in
C_{p_{\ast}}^{+}$. From Proposition \ref{conoGRW}, $\xi_{(t,x)}=\frac{-1}%
{\int_{t_{\ast}}^{t}f(r)dr}P_{(t,x)}$. If $w\in\mathcal{S}_{(t,x)}$, then
$g(w,\partial_{t})=0$ and so $w\in T_{x}F$ with $g(P_{x}^{F},w)=0$.
Therefore,
\[
\nabla_{w}\xi=\frac{-1}{\int_{t_{\ast}}^{t}f(r)dr}\nabla_{w}P=-\frac
{f^{\prime}(t)}{f(t)^{2}}w-\frac{1}{f(t)^{2}\int_{t_{\ast}}^{t}\frac{1}%
{f(r)}dr}\nabla_{w}P^{F}.
\]

Now, we use Lemma \ref{lemaJacobi} and that $F$ has constant curvature $k$ to
compute $g_{F}(\nabla_{w}^{F}P^{F},w)$. Take $u\in T_{x_{*}}F$ such that
$exp^{F}_{x_{*}}(d_{F}(x,x_{*})u)=x$. The Jacobi vector field over
$exp_{x_{\ast}}^{F}(s\,d_{F}(x,x_{*})u)$, $0\leq s\leq1$, with $J(0)=0$ and
$J(1)=w$ is given by $J(s)=\varphi(s)W(s)$, where $W$ is parallel with $W(1)=w
$ and
\[
\varphi(s)=\left\{
\begin{array}
[c]{lcc}%
\frac{Sin\left(  \sqrt{k}d_{F}(x,x_{*}) s\right)  }{Sin\left(  \sqrt{k}%
d_{F}(x,x_{*})\right)  } & if & k>0,\\
s & if & k=0,\\
\frac{Sinh\left(  \sqrt{-k} d_{F}(x,x_{*}) s\right)  }{Sinh\left(  \sqrt
{-k}d_{F}(x,x_{*}) \right)  } & if & k<0.
\end{array}
\right.
\]

Therefore, since $d_{F}(x,x_{\ast})=\int_{t_{\ast}}^{t} \frac{1}{f(r)}dr$, we
have
\[
B(w,w)=\left\{
\begin{array}
[c]{lcc}%
\frac{1}{f(t)^{2}}\left(  f^{\prime}(t)+\frac{\sqrt{k}}{Tan\left(  \sqrt
{k}\int_{t_{\ast}}^{t}\frac{1}{f(r)}dr\right)  }\right)  g(w,w) & if & k>0,\\
\frac{1}{f(t)^{2}}\left(  f^{\prime}(t)+\frac{1}{\int_{t_{\ast}}^{t}\frac
{1}{f(r)}dr}\right)  g(w,w) & if & k=0,\\
\frac{1}{f(t)^{2}}\left(  f^{\prime}(t)+\frac{\sqrt{-k}}{Tanh\left(  \sqrt
{-k}\int_{t_{\ast}}^{t}\frac{1}{f(r)}dr\right)  }\right)  g(w,w) & if & k<0.
\end{array}
\right.
\]

\end{proof}

\begin{remark}
\label{conogeodesico} In view of the null second fundamental form of nullcones in
Robertson-Walker spaces, it follows that they can not be totally geodesic,
since the null mean curvature tends to infinity as the coordinate $t$
approaches to $t_{\ast}$. This is a general fact in any Lorentzian manifold (see for example Proposition 2.1 in \cite{QWang}).
\end{remark}

\begin{example}
Since $\mathbb{R}_{1}^{n}$, $\mathbb{S}_{1}^{n}$ and a suitable portion of
$\mathbb{H}_{1}^{n}$ can be expressed as a RW space (see the first column of
Table \ref{tabla1}), the above proposition shows the well-known fact that
nullcones of a Lorentzian manifolds of constant curvature are totally umbilic.
\end{example}

\begin{proposition}
\label{Jacobi}Let $\gamma$ be a null geodesic in a GRW space. If $\gamma$ is
contained in a totally umbilic nullcone $L$, and $J$ is a Jacobi vector field
with $J\in\mathcal{S}$, then it holds
\[
J^{\prime\prime}+\frac{Ric(\gamma^{\prime},\gamma^{\prime})}{n-2}J=0.
\]

In particular, if there exists a conjugate point of $\gamma(0)$ along $\gamma
$, it has maximum multiplicity.
\end{proposition}

\begin{proof}
Using that $L$ is totally umbilic, we have $\nabla_{X}\xi=-\rho X$ for all
$X\in\mathcal{S}$. Since $\xi$ is geodesic, after a suitable affine
reparametrization, $\gamma$ is an integral curve of $\xi$, thus
\[
R_{J\gamma^{\prime}}\gamma^{\prime}=R_{J\xi}\xi=\left(  \xi(\rho)-\rho
^{2}\right)  J.
\]

Using Equation (\ref{Eq1})$,$ we get the result. Finally observe that the
Jacobi operator is proportional to the identity, so if there exists a
conjugate point, then it has maximum multiplicity.
\end{proof}

This proposition is potentially interesting in Cosmology because the
multiplicity in gravitational lens phenomena can be detected by astronomical
observations. On the other hand, after Proposition \ref{conosumbilicos},
conjugate points in Robertson-Walker spaces have maximum multiplicity, see
also \cite{FloSan00}.

We finish this section with a criterion to determine whenever a null
hypersurface given by the graph of a function is contained in a nullcone.

\begin{lemma}
\label{lema1}Let $I\times_{f}F$ be a GRW space and $p_{\ast}=(t_{\ast}%
,x_{\ast})$ a fixed point. Let $\theta\subset F$ with $x_{\ast}$ in the
clausure of $\theta$ and $h:\theta\rightarrow I$ a function. If the
graph of $h$ is a null hypersurface, then it is contained in the local
nullcone at $p_{\ast}$ if and only if $\lim_{x\rightarrow x_{\ast}%
}h(x)=t_{\ast}$ and $\nabla^{F}h$ is proportional to $P^{F}$, the local
position vector field at $x_{\ast}$ in $F$.
\end{lemma}

\begin{proof}
Take a null geodesic $\gamma$ such that $\gamma(0)=p_{\ast}$ and
$\gamma(s)=\left(  \alpha(s),\beta(s)\right)  $. We have
\begin{align*}
g_{F}\left(  \nabla^{F}h,\beta^{\prime}(s)\right)   &  =(h\circ\beta)^{\prime
}(s),\\
\left\vert \nabla^{F}h\right\vert _{F}\left\vert \beta^{\prime}(s)\right\vert
_{F}  &  =\left\vert \alpha^{\prime}(s)\right\vert .
\end{align*}

Using the Cauchy-Schwarz inequality and that locally $C_{p_{\ast}}^{\pm}%
=\exp_{p_{\ast}}(\widehat{C}_{p_{\ast}}^{\pm})$, the graph of $h$ is contained
in the nullcone at $p_{\ast}$ if and only if $\nabla^{F}h$ and $\beta^{\prime
}$ are proportional and $\lim_{x\rightarrow x_{\ast}}h(x)=t_{\ast}$, but
observe that $\beta^{\prime}$ is proportional to $P^{F}$.
\end{proof}

\begin{remark}
If $L$ is a null hypersurface in a GRW space given by the graph of a function
$h$ and $p_{\ast}\in L$, then for any null geodesic $\gamma(s)=\left(
\alpha(s),\beta(s)\right)  $ with $\gamma(0)=p_{\ast}$ we have $h(\beta
(s))\leq\alpha(s)$ and the equality holds if and only if $\gamma$ belongs to
$L$. This implies that, near $p_{\ast}$, it holds $L\subset I^{+}(p_{\ast
})^{c}$. Geometrically, this means that the local nullcone at $p_{\ast}$ is an
extremal null hypersurface near $p_{\ast}$. This result is also true for any
arbitrary Lorentzian manifold. Indeed, given a point $p_{\ast}\in L$ consider
$\Theta$ a normal neighborhood of $p_{\ast}$ and suppose there exists a point
$q\in L\cap I^{+}(p_{\ast},\Theta)$. Take a timelike plane $\Pi\subset
T_{p_{\ast}}M$ with $0,\widehat{q}\in\Pi$, where $\widehat{q}=\exp_{p_{\ast}%
}^{-1}(q)$. The intersection $L\cap\exp_{p_{\ast}}(\Pi)$ is the trace of a
curve from $p_{\ast}$ to $q$ which is null or spacelike in each of its points
and is contained in the timelike surface $\exp_{p_{\ast}}(\Pi)$, but $q\in
I^{+}(p_{\ast},\exp_{p_{\ast}}(\Pi))$, which is a contradiction. Using a past
nullcone sharing a null geodesic of $L$, with vertex near $p_{\ast}$, we can
figure out the situation as $L$ being a sheet between two millstones.
\end{remark}

\section{Umbilic null hypersurfaces}

\label{seccionumbilica}

In this section we prove the main result of this paper. It gives us a
correspondence between totally umbilic null hypersurfaces and twisted
decompositions of the fibre of a GRW space. Thus, it shows that only special
types of GRW spaces can admit totally umbilic null hypersurfaces and it also
provides a method to construct them. First, we need the following lemma.

\begin{lemma}
\label{lemaB} Let $I\times_{f}F$ be a GRW space and $L$ a null hypersurface
given by the graph of a function $h$. Then
\begin{equation}
\xi=-\frac{1}{f\circ h}\partial_{t}-\frac{1}{\left(  f\circ h\right)  ^{3}%
}\nabla^{F}h, \label{campoxi}%
\end{equation}
$\mathcal{S}\simeq \{X\in TF:g_{F}(X,\nabla^{F}h)=0\}$ and the null second fundamental form is
given by
\[
B(X,Y)=\frac{f^{\prime}\circ h}{\left(  f\circ h\right)  ^{2}}g(X,Y)+\frac
{1}{f\circ h}Hess_{h}^{F}(X,Y),
\]
for all $X,Y\in\mathcal{S}$.
\end{lemma}

\begin{theorem}
\label{umbilicGRW} Let $I\times_{f}F$ be a GRW space. If $L$ is a totally
umbilic null hypersurface, then for each $(t_{0},x_{0})\in L$ there exists a
decomposition of $F$ in a neighborhood of $x_{0}$ as a twisted product with
one dimensional base
\[
\big(J\times S,ds^{2}+\mu(s,z)^{2}g_{S}\big),
\]
where $x_{0}$ is identified with $(0,z_{0})$ for some $z_{0}\in S$ and $L$ is
given by
\[
\{(t,s,z)\in I\times J\times S:s=\int_{t_{0}}^{t}\frac{1}{f(r)}dr\}.
\]

Moreover, if $H$ is the null mean curvature of $L$, then
\[
\mu(s,z)=\frac{f(t_{0})}{f(t)}exp\left(  \int_{0}^{s}\frac{H(t,r,z)f(t)^{2}%
}{n-2}dr\right)
\]
for all $(t,s,z)\in L$.

Conversely, if $F$ admits a twisted decomposition in a neighborhood of $x_{0}$
as above, then $L=\{(t,s,z)\in I\times J\times S:s=\int_{t_{0}}^{t}\frac
{1}{f(r)}dr\}$ is a totally umbilic null hypersurface with null mean curvature
\[
H=\frac{n-2}{f(t)^{2}}\left(  f^{\prime}(t)+\frac{\mu_{s}(s,z)}{\mu
(s,z)}\right)  .
\]

\end{theorem}

\begin{proof}
Suppose that $L$ is given by the graph of certain map $h:\theta\rightarrow
I$ in a neighborhood of $(t_{0},x_{0})$ and $B=\rho g$. From Formula
(\ref{remarkh}) and Lemma \ref{lemaB} it holds
\begin{align*}
\nabla_{\nabla^{F}h}^{F}\nabla^{F}h  &  =(f\circ h)(f^{\prime}\circ
h)\nabla^{F}h,\\
\nabla_{v}^{F}\nabla^{F}h  &  =\left(  \rho(f\circ h)^{3}-(f\circ
h)(f^{\prime}\circ h)\right)  v
\end{align*}
for all $v\perp\nabla^{F}h$. Since $\left\vert \nabla^{F}h\right\vert
_{F}=f\circ h$, we have that $\left\vert \nabla^{F}h\right\vert _{F}$ is
constant through the level hypersurfaces of $h$. Thus, if we call $E=\frac
{1}{\left\vert \nabla^{F}h\right\vert _{F}}\nabla^{F}h$, it is easy to show
that
\begin{align*}
\nabla_{E}^{F}E  &  =0,\\
\nabla_{v}^{F}E  &  =\big(\rho\left(  f\circ h\right)  ^{2}-(f^{\prime}\circ
h)\big)v
\end{align*}
for all $v\perp\nabla^{F}h$. From these equations it follows that $E$ is
closed and
\begin{equation}
\left(  L_{E}g_{F}\right)  (v,w)=2\big(\rho\left(  f\circ h\right)
^{2}-(f^{\prime}\circ h)\big)g_{F}(v,w) \label{eq1teo}%
\end{equation}
for all $v,w\in TF$ with $v,w\perp\nabla^{F}h$.

The following argument is local, so for simplicity we can suppose without loss
of generality that $E$ is complete. Call $S_{z}$ the leaf of $E^{\perp}$
through $z\in F$. Being $E$ closed, the flow $\phi$ of $E$ is foliated, that
is, $\phi_{s}(S_{z})=S_{\phi_{s}(z)}$ for all $z\in L$ and $s\in\mathbb{R}$.
Using this, it is easy to check that $\phi:\mathbb{R}\times S_{x_{0}%
}\rightarrow F$ is onto and a local diffeomorphism. Moreover, from Equation
(\ref{eq1teo}), $\phi_{s}:S_{x_{0}}\rightarrow S_{\phi_{s}(x_{0})}$ is a
conformal diffeomorphism with conformal factor
\[
exp\left(  2\int_{0}^{s}\left(  \rho(\phi_{r}(z))f(h(\phi_{r}(z)))^{2}%
-f^{\prime}(h(\phi_{r}(z)))\right)  dr\right)
\]
and it follows that $\phi^{\ast}(g_{F})=ds^{2}+\mu^{2}g|_{S_{x_{0}}}$, being
\[
\mu(s,z)=exp\left(  \int_{0}^{s}\left(  \rho(\phi_{r}(z))f(h(\phi_{r}%
(z)))^{2}-f^{\prime}(h(\phi_{r}(z)))\right)  dr\right)  .
\]

Since $E$ is identified with $\partial_{s}$, in this decomposition $h$ only
depends on $s$ and $h^{\prime}(s)>0$. Thus, from Equation (\ref{eqhyperluzGRW}%
) we have that $h(s)=c^{-1}(s)$ being $c(t)=\int_{t_{0}}^{t}\frac{1}{f(r)}dr$.
Moreover, $\left(  \ln f\circ h\right)  ^{\prime}=f^{\prime}\circ h$, so the
above expression for $\mu$ can be written as
\[
\mu(s,z)=\frac{f(t_{0})}{f(h(s))}exp\left(  \int_{0}^{s}\rho(\phi
_{r}(z))f(h(r))^{2}dr\right)  .
\]

For the converse, Equation (\ref{eqhyperluzGRW}) can be directly checked and
applying Lemma \ref{lemaB} we get the result.
\end{proof}

Observe that $\mu(0,z)=1$ for all $z\in S$. Moreover, this theorem can be
applied to any null surface in a three dimensional GRW space, since they are
always totally umbilic.

\begin{example}
Totally geodesic null hypersurfaces in $\mathbb{R}_{1}^{n}$ are given by null
hyperplanes. In $\mathbb{S}_{1}^{n}$ and $\mathbb{H}_{1}^{n}$ we can obtain
totally geodesic null hypersurfaces intersecting them with null planes through
the origin of $\mathbb{R}_{1}^{n+1}$ and $\mathbb{R}_{2}^{n+1}$ respectively.
Table \ref{tabla1}  shows how Theorem \ref{umbilicGRW} is fulfilled in these
particular cases. In this table
we call
 $A(s)=2\arg\tanh\left(\tan\left(\frac{s}{2}\right)\right)$,
 $B(s)=2\arg\tan\left(\tanh\left(\frac{s}{2}\right)\right)$
 and $\mathbb{H}$ the portion of $\mathbb{H}^n_1$ given by $\left(-\frac{\pi}{2},\frac{\pi}{2}\right)\times_{\cos(t)}\mathbb{H}^{n-1}$.
\end{example}

\begin{table}[h]
\caption{Totally geodesic null hypersurfaces in space forms.}
\label{tabla1}
\begin{tabular}{lll}
\hline\noalign{\smallskip}
\textbf{Space form} & \textbf{Fibre} & \textbf{Tot. geod. null hypersurface} \\
\noalign{\smallskip}\hline\noalign{\smallskip}
$\mathbb{R}^n_1=\mathbb{R}\times\mathbb{R}^{n-1}$  & $\mathbb{R}\times\mathbb{R}^{n-2}$ & $\left\{(s,s,z)\in \mathbb{R}\times\mathbb{R}\times\mathbb{R}^{n-2}\right\}$  \\
$\mathbb{S}^{n}_1=\mathbb{R}\times_{\cosh(t)}\mathbb{S}^{n-1}$  & $\left(-\frac{\pi}{2},\frac{\pi}{2}\right)\times_{\cos(s)}\mathbb{S}^{n-2}$ & $\Big\{\left(A(s),s,z\right)\in\mathbb{R}\times \left(-\frac{\pi}{2},\frac{\pi}{2}\right)\times\mathbb{S}^{n-2}\Big\}$ \\
$\mathbb{H}$ & $\mathbb{R}\times_{\cosh(s)}\mathbb{H}^{n-2}$ & $\Big\{\left(B(s),s,z\right)\in \left(-\frac{\pi}{2},\frac{\pi}{2}\right)\times\mathbb{R}\times\mathbb{H}^{n-2}\Big\}$ \\
\end{tabular}
\end{table}

\begin{example}
\label{ejemplodescomposicionconos} From Proposition \ref{conosumbilicos} and
Theorem \ref{umbilicGRW}, nullcones in a Robertson-Walker space induce a
twisted decomposition of the fibre. Indeed. Consider, for example,
$M=I\times_{f}\mathbb{S}^{n}$ and fix $(t_{0},x_{0})\in C_{(t_{\ast},x_{\ast
})}^{+}$. If we call $\delta=d_{\mathbb{S}^{n}}(x_{\ast},x_{0})$, then from
Proposition \ref{conoGRW} we have $\int_{t_{\ast}}^{t_{0}}\frac{1}%
{f(r)}dr=\delta$. Take the local decomposition of $\mathbb{S}^{n}$ given by
$\left(  0,\pi\right)  \times_{\sin(u)}\mathbb{S}^{n-1}$, where $d_{\mathbb{S}%
^{n}}(x_{\ast},x)=u$. In this decomposition, the point $x_{0}$ is identified
with $(\delta,z_{0})$ for some $z_{0}\in\mathbb{S}^{n-1}$. Thus, if we call
$s=u-\delta$ we get the decomposition $\left(  -\delta,\pi-\delta\right)
\times_{\sin(s+\delta)}\mathbb{S}^{n-1}$, where $x_{0}$ is identified with
$(0,z_{0})$ and $C_{(t_{\ast},x_{\ast})}^{+}$ is given by
\[
\{(t,s,z)\in I\times\left(  -\delta,\pi-\delta\right)  \times\mathbb{S}%
^{n-1}:\int_{t_{0}}^{t}\frac{1}{f(r)}dr=s\}
\]
as Theorem \ref{umbilicGRW} asserts.

Analogously, nullcones in $I\times_{f}\mathbb{R}^{n}$ induce the decomposition
of $\mathbb{R}^{n}$ given by $\left(  -\delta,\infty\right)  \times_{s+\delta
}\mathbb{S}^{n-1}$, whereas nullcones in $I\times_{f}\mathbb{H}^{n}$ induce
the decomposition of $\mathbb{H}^{n}$ given by $\left(  -\delta,\infty\right)
\times_{\sinh(s+\delta)}\mathbb{S}^{n-1}$.
\end{example}

Given a totally umbilic null hypersurface $L$, we can construct another one,
which we call dual of $L$, simply by changing the sign of the parameter in
the base of the twisted decomposition of the fibre induced by $L$.

\begin{corollary}
\label{hipersuperficiedual} Let $I\times_{f}F$ be a GRW space and $L$ a
totally umbilic null hypersurface. For each $(t_{0},x_{0})\in L$ we can
construct another totally umbilic null hypersurface $\widetilde{L}%
_{(t_{0},x_{0})}$, which we call dual of $L$ through $(t_{0},x_{0})$.

Specifically, if $L$ induces a twisted decomposition $J\times_{\mu}S$ of $F$
in a neighborhood of $x_{0}$ where $L$ is $\{(t,s,z)\in I\times J\times
S:s=\int_{t_{0}}^{t}\frac{1}{f(r)}dr\}$, then $\widetilde{L}_{(t_{0},x_{0})}$
is given by
\[
\{(t,s,z)\in I\times J\times S:s=\int_{t}^{t_{0}}\frac{1}{f(r)}dr\}
\]
and its null mean curvature is
\[
\widetilde{H}=\frac{n-2}{f(t)^{2}}\left(  f^{\prime}(t)-\frac{\mu_{s}%
(s,z)}{\mu(s,z)}\right)  .
\]

\end{corollary}

\begin{proof}
Suppose $J=(-\varepsilon,\varepsilon)$ and consider the coordinate change
\[
\psi:I\times_{f}\left(  J\times_{\mu(s,z)}S\right)  \rightarrow I\times
_{f}\left(  J\times_{\mu(-u,z)}S\right)
\]
given by the isometry $\psi(t,s,z)=(t,-s,z)$. Using Theorem \ref{umbilicGRW},
in the codomain of $\psi$, the twisted decomposition $du^{2}+\mu
(-u,z)^{2}g_{S}$ induces the totally umbilic null hypersurface given by
$\{(t,u,z)\in I\times J\times S:u=\int_{t_{0}}^{t}\frac{1}{f(r)}dr\}$. The
inverse image of this hypersurface is $\widetilde{L}_{(t_{0},x_{0}%
)}=\{(t,s,z):s=\int_{t}^{t_{0}}\frac{1}{f(r)}dr\}$. Its null mean curvature
can be easily computed using Lemma \ref{lemaB}.
\end{proof}

\begin{example}
\label{Ej2}Consider $\mathbb{R}_{1}^{n}=\mathbb{R}\times\mathbb{R}^{n-1}$ and
$L=C_{(0,0)}^{+}$. Fixed $(t_{0},x_{0})\in L$, from Example
\ref{ejemplodescomposicionconos} and the above corollary, the dual
hypersurface through $(t_{0},x_{0})$ is given by
\[
\{(t,s,z)\in\mathbb{R}\times(-t_{0},\infty)\times\mathbb{S}^{n-2}%
:s=t_{0}-t\}=C_{(2t_{0},0)}^{-}.
\]

Consider now $\mathbb{S}_{1}^{n}=\mathbb{R}\times_{\cosh(t)}\mathbb{S}^{n-1}$
and $L=C_{(0,x_{\ast})}^{+}$. As before, fixed $(t_{0},x_{0})\in L$,
$\widetilde{L}_{(t_{0},x_{0})}$ is given by
\[
\{(t,s,z)\in\mathbb{R}\times(-\delta,\pi-\delta)\times\mathbb{S}^{n-2}%
:s=\int_{t}^{t_{0}}\frac{1}{\cosh(r)}dr\},
\]
where $\delta=\int_{0}^{t_{0}}\frac{1}{\cosh(r)}dr$ and $s=d_{\mathbb{S}%
^{n-1}}(x_{\ast},x)-\delta$. Take $t_{c}>0$ such that
\[
\int_{0}^{t_{c}}\frac{1}{\cosh(r)}dr=\int_{t_{c}}^{\infty}\frac{1}{\cosh
(r)}dr=\frac{\pi}{4}.
\]

If $t_{0}<t_{c}$, then there exists $t_{s}$ such that $\delta=\int_{t_{0}%
}^{t_{s}}\frac{1}{\cosh(r)}dr$ and using Proposition \ref{conoGRW}, it is easy
to show that $\widetilde{L}_{(t_{0},x_{0})}=C_{(t_{s},x_{\ast})}^{-}$.

If $t_{c}<t_{0}$, using that $d_{\mathbb{S}^{n-1}}(x^{\ast},x_{\ast})=\pi$,
where $x^{\ast}$ is the antipodal of $x_{\ast}$, we write $s=\pi-d_{\mathbb{S}%
^{n-1}}(x^{\ast},x)-\delta$. If we take $t_{l}$ such that $\pi-\delta
=\int_{t_{l}}^{t_{0}}\frac{1}{\cosh(r)}dr$, then it follows that
$\widetilde{L}_{(t_{0},x_{0})}=C_{(t_{l},x^{\ast})}^{+}$.

Finally, suppose that $t_{0}=t_{c}$. In this case,
\[
\widetilde{L}_{(t_{0},x_{0})}=\{(t,s,z)\in\mathbb{R}\times\left(  -\frac{\pi
}{4},\frac{3\pi}{4}\right)  \times\mathbb{S}^{n-2}:s-\frac{\pi}{4}=-2\arg
\tan\left(  \tanh\left(  \frac{t}{2}\right)  \right)  \}.
\]

Reparametrizing the $s$ coordinate, it follows that it is the totally geodesic
null hypersurface given in Table \ref{tabla1}.
\end{example}

The above example shows that the dual construction in Minkowski space is a
time reflection, in the sense that the dual of a future nullcone is a past
nullcone. However, in the De Sitter space it is more involved, since the dual
of a future nullcone through a given point on it can be a past nullcone, a
totally geodesic null hypersurface, or even another future nullcone, depending on
the situation of the given point.

As an immediate corollary of Theorem \ref{umbilicGRW}, we can give the
following obstruction to the existence of totally umbilic (geodesic) null hypersurfaces.

\begin{corollary}
\label{corobs} If the fibre of a GRW space does not admit any local
decomposition as a twisted (warped) product with one dimensional base, then it
does not exist any totally umbilic (geodesic) null hypersurface.
\end{corollary}

If $\mathcal{K}(\Pi)\neq0$ for any null plane at a point $p$ in an arbitrary
Lorentzian manifold, then Equation (\ref{Eq1}) implies that it does not exist
any totally geodesic null hypersurface through $p$. In a GRW space
$I\times_{f}F$, the null sectional curvature of a null plane $\Pi
=span(v,u=-\partial_{t}+w)$, where $v,w\in TF$ are unitary and orthogonal, is
given by
\[
\mathcal{K}_{u}(\Pi)=\frac{K^{F}\left(  span(v,w)\right)  +f^{\prime
2}-ff^{\prime\prime}}{f^{2}}.
\]

Therefore, this obstruction to the existence of totally geodesic hypersurfaces
involves both the curvature of the fibre and the warping function. However,
the obstruction given in Corollary \ref{corobs} is more general because it
includes totally umbilic null hypersurfaces and only depends on the fibre.

\begin{example}
\label{productoesferas} In a Riemannian twisted product manifold $J\times
_{\mu}S$, the sectional curvature of any plane containing $\partial_{s}$ is
$\frac{-1}{\mu}Hess_{\mu}(\partial_{s},\partial_{s})$. Therefore,
$\mathbb{S}^{2}\times\mathbb{S}^{2}$ does not admit any local twisted product
decomposition as above, since for any vector we can find two planes
containing it with different sectional curvatures. Applying Corollary
\ref{corobs}, in a GRW space $I\times_{f}\left(  \mathbb{S}^{2}\times
\mathbb{S}^{2}\right)  $ there are not totally umbilic null hypersurfaces.
\end{example}

\begin{example}
Consider the twisted product $F=\mathbb{R}\times_{\mu}\mathbb{R}^{n}$, where
$\mu(s,z)=e^{s}+|z|^{2}$. A curvature analysis as before shows that it does
not admit another local decomposition as a twisted nor warped product with a
one dimensional base. Therefore, from Corollary \ref{hipersuperficiedual}, in a
GRW space $I\times_{f}F$ there are exactly two totally umbilic null
hypersurface through each point and using Corollary \ref{corobs}, it does not
have any totally geodesic null hypersurface.
\end{example}

Theorem \ref{umbilicGRW} does not hold for timelike nor spacelike
hypersurfaces. Indeed, in $\mathbb{R}\times\mathbb{S}^{2}\times\mathbb{S}^{2}$
we can find totally geodesic timelike or spacelike hypersurfaces, although, as
it was shown in Example \ref{productoesferas}, $\mathbb{S}^{2}\times
\mathbb{S}^{2}$ does not admit local decompositions as a twisted product.
However, it is known that a totally umbilic timelike hypersurface in a GRW
space must be itself a GRW space, \cite{GutOle09}.

The following lemma gives us another characterization of an open set of a
nullcone near its vertex.

\begin{lemma}
\label{lemacaractcono} Let $I\times_{f}F$ be a GRW space with $F$ complete, take
 $(t_{0},x_{0})\in I\times F$ and suppose that $F$ decomposes in a neighborhood of $x_{0}$ as
a twisted product%
\[
\left(  (a,b)\times S,ds^{2}+\mu(s,z)^{2}g_{S}\right)  ,
\]
where $-\infty<a<0<b\leq\infty$ ($-\infty\leq a<0<b<\infty$), $S$ is connected and
$x_{0}$ is identified with $(0,z_{0})$ for some $z_{0}\in S$. The null
hypersurface
\[
L=\{(t,s,z)\in I\times(a,b)\times S:s=\int_{t_{0}}^{t}\frac{1}{f(r)}dr\}
\]
is contained in a future (past) local nullcone if and only if

\begin{enumerate}
\item $\lim_{s\rightarrow a^{+}}\mu(s,z)=0$ for all $z\in S$ ($\lim
_{s\rightarrow b^{-}}\mu(s,z)=0$ for all $z\in S$).

\item It exists $t_{*}\in I$ with $\int^{t_{*}}_{t_{0}}\frac{1}{f(r)}dr=a$
$\left(  \int_{t_{0}}^{t_{*}}\frac{1}{f(r)}dr=b\right)  $.
\end{enumerate}
\end{lemma}

\begin{proof}
Suppose that (1) and (2) hold. Since the integral curves of $\partial_{s}$ are
unitary geodesics and $F$ is complete, it exists $\lim_{s\rightarrow a^{+}%
}(s,z)$ for all $z\in S$. Fix $z,z^{\prime}\in S$ two distinct points such
that there exists $\sigma(r)$ a unitary geodesic in $S$ with $\sigma(0)=z$ and
$\sigma(d)=z^{\prime}$. If we call $\gamma_{s}(r)=(s,\sigma(r))$, then%
\[
d_{F}(\gamma_{s}(0),\gamma_{s}(d))\leq length(\gamma_{s})=\int_{0}^{d}%
\mu(s,\sigma(r))dr
\]
and so $lim_{s\rightarrow a^{+}}d_{F}(\gamma_{s}(0),\gamma_{s}(d))=0$.
Therefore, since $S$ is connected, $\lim_{s\rightarrow a^{+}}(s,z)$ is the
same for all $z\in S$, say $x_{\ast}\in F$, and the integral curves of
$\partial_{s}$ are radial geodesic from $x_{\ast}$. Thus, $d_{F}(x_{\ast
},(s,z))=s-a$ and given $(t,x)\in L$ we have%
\[
d_{F}(x_{\ast},x)=\int_{t_{0}}^{t}\frac{1}{f(r)}dr-a=\int_{t_{\ast}}^{t}%
\frac{1}{f(r)}dr.
\]

By Proposition \ref{conoGRW}, $L$ is contained in the nullcone $C_{(t_{\ast
},x_{\ast})}^{+}$.

The converse is straightforward.
\end{proof}

\begin{theorem}
\label{teoremaintegralf} Any totally umbilic null hypersurface in a
Robertson-Walker space $I\times_{f}\mathbb{S}^{n-1}$ ($n>3$) with
\begin{equation}
\int_{I}\frac{1}{f(r)}dr>\pi\label{integraldef}%
\end{equation}
is an open set of a nullcone. In particular, it cannot exist totally geodesic
null hypersurfaces.
\end{theorem}

\begin{proof}
Let $L$ be a totally umbilic null hypersurface and take $(t_{0},x_{0})\in L$.
Using Theorem \ref{umbilicGRW}, $\mathbb{S}^{n-1}$ can be decomposed in a
neighborhood of $x_{0}$ as a twisted product. Since $\mathbb{S}^{n-1}$ is
Einstein, this decomposition is actually a warped product \cite{GarciaRio},
and it is easy to show that it is
\[
\left(  -\frac{\pi}{2}-\theta,\frac{\pi}{2}-\theta\right)  \times_{\mu
}\mathbb{S}^{n-2}(|\cos(\theta)|),
\]
where $\theta\in(-\frac{\pi}{2},\frac{\pi}{2})$, $\mu(s)=\frac{\cos(s+\theta
)}{\cos(\theta)}$ and $x_{0}$ is identified with $(0,z_{0})$ for some
$z_{0}\in\mathbb{S}^{n-2}$.

Using (\ref{integraldef}), it exists $t_{\ast}\in I$ such that $\int_{t_{0}%
}^{t_{\ast}}\frac{1}{f(r)}dr=\frac{\pi}{2}-\theta$ or $\int_{t_{0}}^{t_{\ast}%
}\frac{1}{f(r)}dr=-\frac{\pi}{2}-\theta$ and applying the above lemma, $L$ is
contained in a lightcone. The last claim follows from Remark
\ref{conogeodesico}.
\end{proof}

The condition (\ref{integraldef}) can not be sharpened. For example, in
$\mathbb{S}_{1}^{n}=\mathbb{R}\times_{\cosh(t)}\mathbb{S}^{n-1}$ there are
totally geodesic null hypersurfaces which evidently are not contained in a nullcone.

We can also get the following immediate corollaries.

\begin{corollary}
Nullcones are the unique totally umbilic null hypersurfaces in the closed
Friedmann Cosmological model.
\end{corollary}

\begin{corollary}
\label{corpd} Any totally umbilic null hypersurface in $\mathbb{R}\times\mathbb{S}%
^{n-1}$ ($n>3$) is contained in a nullcone.
\end{corollary}

Recall that both, Friedmann models and $\mathbb{R}\times\mathbb{S}^{n-1}$, can not
possess totally geodesic null hypersurfaces due to Equation (\ref{Eq1}) and
Lemma 5.2 of \cite{GutOle09}.

In \cite{Akivis} it is shown that totally umbilic null hypersurfaces in a
Lorentzian manifold of constant curvature are contained in nullcones. The
proof is based on their following claim: in a Lorentzian manifold any totally
umbilic null hypersurfaces with zero null sectional curvature is contained in
a nullcone. But the example below shows that this is not true in general.

\begin{example}
\label{Ej1}Let $Q\times_{r}\mathbb{S}^{2}$ be the Kruskal spacetime, \cite{O}.
The hypersurface $L_{u_{0}}=\{(u,v,x)\in Q\times\mathbb{S}^{2}:u=u_{0}\}$ is
totally umbilic and null. Moreover, if $\Pi$ is a null tangent plane to
$L_{u_{0}}$, then it is spanned by $\partial_{v}$ and $w\in T\mathbb{S}^{2}$,
so
\[
\mathcal{K}_{\partial_{v}}(\Pi)=-\frac{Hess_{r}(\partial_{v},\partial_{v})}%
{r}=0,
\]
but $L_{u_{0}}$ is not contained in a nullcone.
\end{example}

However, under completeness hypothesis it seems that the above claim is true. For clarity, we give an alternative proof of the following
result using the technique presented in this paper.

\begin{theorem}
\label{Teor1}Any totally umbilic null hypersurface in a complete space of
constant curvature and dimension greater than three is totally geodesic or is
contained in a nullcone.
\end{theorem}

\begin{proof}
We can suppose that $M$ is $\mathbb{R}_{1}^{n}$, $\mathbb{S}_{1}^{n}$ or
$\mathbb{H}_{1}^{n}$.

Suppose first that $M=\mathbb{R}_{1}^{n}=\mathbb{R}\times\mathbb{R}^{n-1}$,
and $L$ is a totally umbilic and non totally geodesic null hypersurface in
$M$. From Theorem \ref{umbilicGRW} and \cite{GarciaRio}, it induces a
decomposition of $\mathbb{R}^{n-1}$ as a warped product, but since $L$ is not
totally geodesic, the only possible decomposition is $\left(  -\frac{1}%
{\theta},\infty\right)  \times_{\theta s+1}\mathbb{S}^{n-2}\left(  \frac
{1}{\theta}\right)  $ for $\theta>0$ or $\left(  -\infty,-\frac{1}{\theta
}\right)  \times_{\theta s+1}\mathbb{S}^{n-2}\left(  -\frac{1}{\theta}\right)
$ for $\theta<0$. Applying Lemma \ref{lemacaractcono}, $L$ is contained in a nullcone.

Suppose now that $M=\mathbb{S}_{1}^{n}=\mathbb{R}\times_{\cosh(t)}%
\mathbb{S}^{n-1}$. Without loss of generality, we can suppose that
$(0,x_{0})\in L$ for some $x_{0}\in\mathbb{S}^{n-1}$. As in the proof of
Theorem \ref{teoremaintegralf}, there is a decomposition of $\mathbb{S}^{n-1}$
as
\[
\left(  -\frac{\pi}{2}-\theta,\frac{\pi}{2}-\theta\right)  \times_{\mu
}\mathbb{S}^{n-2}(|\cos(\theta)|),
\]
where $\theta\in(-\frac{\pi}{2},\frac{\pi}{2})$ and $\mu(s)=\frac
{\cos(s+\theta)}{\cos(\theta)}$. If $\theta=0$, then $L$ is totally geodesic
(see Table \ref{tabla1}). If $\theta\neq0$, using Lemma \ref{lemacaractcono},
$L$ is contained in a nullcone.

Finally, we consider $M=\mathbb{H}_{1}^{n}\subset\mathbb{R}_{2}^{n+1}$. Since
it does not admit a global decomposition as a RW space, a little more work
must be done in this case. We can suppose that $L$ intersects an open set of
$\mathbb{H}_{1}^{n}$ isometric to $\left(  -\frac{\pi}{2},\frac{\pi}%
{2}\right)  \times_{\cos t}\mathbb{H}^{n-1}$ and $(0,x_{0})\in L$ for some
$x_{0}\in\mathbb{H}^{n-1}$. Applying Theorem \ref{umbilicGRW} and
\cite{GarciaRio}, there is a decomposition of $\mathbb{H}^{n-1}$ as a warped
product $J\times_{\mu}S$ in a neighborhood of $x_{0}$ and $L$ is given by%
\[
\{(2\arg\tan\left(  \tanh\frac{s}{2}\right)  ,s,z):s\in J,z\in S\}.
\]

The decomposition of $\mathbb{H}^{n-1}$ can be of three different types.

\begin{itemize}
\item $J\times_{\frac{\sinh(s+\theta)}{\sinh\theta}}\mathbb{S}^{n-2}\left(
|\sinh\theta|\right)  $ where $J=\left(  -\theta,\infty\right)  $ if
$\theta>0$ and $J=\left(  -\infty,-\theta\right)  $ if $\theta<0$. In this
case we can apply Lemma \ref{lemacaractcono}.

\item $\mathbb{R}\times_{\frac{\cosh(s+\theta)}{\cosh\theta}}\mathbb{H}%
^{n-2}(\cosh\theta)$ where $\theta\in\mathbb{R}$. If $\theta=0$, then $L$ is
totally geodesic (see Table \ref{tabla1}), so we suppose $\theta\neq0$. The
map
\[
\Phi:\left(  -\frac{\pi}{2},\frac{\pi}{2}\right)  \times_{\cos t}\left(
\mathbb{R}\times_{\frac{\cosh(s+\theta)}{\cosh\theta}}\mathbb{H}^{n-2}%
(\cosh\theta)\right)  \rightarrow\mathbb{R}_{2}^{n+1}%
\]
given by%
\[
\Phi(t,s,z)=\left(  \cos t\sinh(s+\theta),\frac{\cos t\cosh(s+\theta)}%
{\cosh\theta}z,\sin t\right)
\]
is an isometric embedding into $\mathbb{H}_{1}^{n}$. If we call $t=2\arg
\tan\left(  \tanh\frac{s}{2}\right)  $, then it holds $\cos t\cosh s=1$ and
$\sin t=\tanh s$ and it is easy to show that
\begin{align*}
&  \Phi(L)=\\
&  \left\{  (\sinh\theta+\cosh\theta\tanh s,(1+\tanh\theta\tanh s)z,\tanh
s):s\in\mathbb{R},z\in\mathbb{H}^{n-2}\right\}
\end{align*}
is contained in the nullcone of $\mathbb{H}_{1}^{n}$ at $\left(  -\frac
{1}{\sinh\theta},0,\ldots,0,-\frac{1}{\tanh\theta}\right)  $.

\item $\mathbb{R}\times_{e^{s}}\mathbb{R}^{n-2}$. The map
\[
\Psi:\left(  -\frac{\pi}{2},\frac{\pi}{2}\right)  \times_{\cos t}\left(
\mathbb{R}\times_{e^{s}}\mathbb{R}^{n-2}\right)  \rightarrow\mathbb{R}%
_{2}^{n+1}%
\]
given by
\begin{align*}
&  \Psi(t,s,z)\\
&  =\left(  e^{s}\cos t\cdot z,\frac{\cos t\left(  e^{s}\left(  1-|z|^{2}%
\right)  -e^{-s}\right)  }{2},\frac{\cos t\left(  e^{s}\left(  1+|z|^{2}%
\right)  +e^{-s}\right)  }{2},\sin t\right)
\end{align*}
is an isometric embedding. As before, we have
\begin{align*}
&  \Psi(L)\\
&  =\left\{  \left(  \frac{e^{s}\cdot z}{\cosh s},\tanh s-\frac{e^{s}|z|^{2}%
}{2\cosh s},1+\frac{e^{s}|z|^{2}}{2\cosh s},\tanh s\right)  :s\in
\mathbb{R},z\in\mathbb{R}^{n-2}\right\}  ,
\end{align*}
which is contained in the nullcone of $\mathbb{H}_{1}^{n}$ at $\left(
0,\ldots,0,-1,1,-1\right)  $.
\end{itemize}
\end{proof}

\begin{example}
Since $\mathbb{H}^{n-1}$ qualitatively admits different decompositions, we can
show that an analogous result of Corollary \ref{corpd} replacing
$\mathbb{S}^{n-1}$ with $\mathbb{H}^{n-1}$ is not true. In fact, from Theorem
\ref{umbilicGRW} the null hypersurface given by $L=\{(s,s,z)\}$ in
$\mathbb{R}\times\left(  \mathbb{R}\times_{e^{s}}\mathbb{R}^{n-2}\right)
\subset\mathbb{R}\times\mathbb{H}^{n-1}$ is totally umbilic but it is
evidently not contained in a nullcone.
\end{example}

\section{Standard static spaces}

\label{seccionestatico}

Given $I\subset\mathbb{R}$, $\left(  F,g_{F}\right)  $ a Riemannian manifold
and $\phi\in C^{\infty}(F)$ a positive function, the manifold $F\times I$
furnished with the Lorentzian metric $g^{\ast}=g_{F}-\phi^{2}dt^{2}$ is called
a standard static space and is denoted by $F\times_{\phi}I$. Totally umbilic
null hypersurfaces are preserved if we apply a conformal transformation to get
a GRW space with constant warping function. In general, we have the following.

\begin{lemma}
\label{cambioconforme} Let $(M,g^{\ast})$ be a Lorentzian manifold with $\dim
M=n$, $\phi\in C^{\infty}(M)$ a positive function and $g=\frac{1}{\phi^{2}%
}g^{\ast}$. If $L$ is a null hypersurface in $(M,g^{\ast})$ and $B^{\ast}$ its
null second fundamental form respect to a fixed null vector field $\xi
\in\mathfrak{X}(L)$, then $L$ is a null hypersurface in $(M,g)$ with null second
fundamental form $B$ respect to $\xi$ given by
\[
B=\frac{1}{\phi^{2}}\left(  B^{\ast}+\xi\left(  \ln\phi\right)  g^{\ast
}\right)  .
\]

In particular, if $L$ is totally umbilic in $(M,g^{\ast})$ with null mean
curvature $H^{\ast}$, then it is also totally umbilic in $(M,g)$ with null
mean curvature $H=H^{\ast}+(n-2)\xi(\ln\phi)$.
\end{lemma}

\begin{proof}
Just use that $\nabla_{U}V=\nabla^{\ast}_{U}V+\frac{1}{\phi}\left(  g^{\ast
}(U,V)\nabla^{\ast} \phi-U(\phi)V-V(\phi)U\right)  $.
\end{proof}

\begin{theorem}
\label{estatico} Let $F\times_{\phi}I$ be a $n$-dimensional standard static
space. If $L$ is a totally umbilic null hypersurface, then for each
$(x_{0},t_{0})\in L$ there exists a local decomposition of $\left(  F,\frac
{1}{\phi^{2}}g_{F}\right)  $ in a neighborhood of $x_{0}$ as a twisted product
with one dimensional base
\[
\left(  J\times S,ds^{2}+\mu(s,z)^{2}g_{S}\right)  ,
\]
where $x_{0}$ is identified with $(0,z_{0})$ for some $z_{0}\in S$ and $L$ is
given by%
\[
\{(s,z,s+t_{0})\in J\times S\times I\}.
\]

Moreover, if $H^{\ast}$ is the null mean curvature of $L$, then
\[
\mu(s,z)=\frac{\phi(0,z)}{\phi(s,z)}\exp\left(  \int_{0}^{s}\frac{H^{\ast
}(r,z,r+t_{0})}{n-2}dr\right)  .
\]

Conversely, if $\left(  F,\frac{1}{\phi^{2}}g_{F}\right)  $ admits a twisted
decomposition in a neighborhood of $x_{0}$ as above, then%
\[
L=\{(s,z,s+t_{0})\in J\times S\times I\}
\]
is a totally umbilic null hypersurface with null mean curvature
\begin{equation}
H^{\ast}=(n-2)\frac{d}{ds}\ln\left(  \mu\phi\right)
.\label{curvaturamediaestatico1}%
\end{equation}

\end{theorem}

\begin{proof}
Take the conformal metric $g=\frac{1}{\phi^{2}}g^{\ast}$ and apply the above
lemma and Theorem \ref{umbilicGRW}. For this, take into account that from
Equation (\ref{campoxi}), $\xi\left(  \ln\phi\right)  =-\partial_{s}\left(
\ln\phi\right)  $.
\end{proof}

\begin{remark}
\label{hipersuperficiedualestatico} As in Corollary \ref{hipersuperficiedual},
if $L$ is a totally umbilic null hypersurface in a standard static space, we
can construct another totally umbilic null hypersurface through each point
$(x_{0},t_{0})\in L$ which we call $\widetilde{L}_{(x_{0},t_{0})}$, the dual
of $L$ through $(x_{0},t_{0})$. In fact, if $L$ induces a twisted
decomposition of $\left(  F,\frac{1}{\phi^{2}}g_{F}\right)  $ in a
neighborhood of $x_{0}$ where $L$ is given by $\{(s,z,s+t_{0})\}$, then
$\widetilde{L}_{(x_{0},t_{0})}$ is given by $\{(s,z,-s+t_{0})\}$. From
Corollary \ref{hipersuperficiedual} and Lemma \ref{cambioconforme}, its null
mean curvature is
\begin{equation}
\widetilde{H}^{\ast}=(n-2)\frac{d}{ds}\ln\left(  \frac{\phi}{\mu}\right)  .
\label{curvaturamediaestatico2}%
\end{equation}

\end{remark}

Now, we consider the family of standard static spacetimes given by
\begin{equation}
\left(  I\times\mathbb{S}^{2}\times\mathbb{R},\frac{1}{h(r)}dr^{2}+r^{2}%
g_{0}-h(r)dt^{2}\right)  , \label{familiaestaticos}%
\end{equation}
where $I\subset\mathbb{R}$, $h\in C^{\infty}(I)$ is a positive function and
$g_{0}$ is the canonical metric on $\mathbb{S}^{2}$. This family includes
important examples of spacetimes. If $h(r)=1-\frac{m^{2}}{r}+\frac{c^{2}%
}{r^{2}}$ for certain constant $m$ and $c$, then we get the
Reissner-Nordstr\"om spacetime (the Schwarzschild exterior in the case $c=0$)
and if $h(r)=1-\frac{m^{2}}{r}+kr^{2}$, then we obtain the De
Sitter-Schwarzschild spacetime (Minkowski, De Sitter or anti-De Sitter if
$m=0$ and $k=0$, $k>0$ or $k<0$ respectively).

We first need to know how many different twisted decomposition admits the
spatial part of (\ref{familiaestaticos}) to apply Theorem \ref{estatico} to
these spacetimes.

\begin{lemma}
\label{unicidad} Let $F$ be the warped product $\left(  I\times\mathbb{S}%
^{2},ds^{2}+\mu(s)^{2}g_{0}\right)  $. If there exists a different
decomposition of $F$ as a twisted product in a neighborhood of a point, then
$F$ has constant curvature in this neighborhood.
\end{lemma}

\begin{proof}
A twisted decomposition is characterized by the existence of a unitary, closed
and orthogonally conformal vector field, \cite{GutOle09}. Therefore, if there
exists a different decomposition as a twisted product in a neighborhood
$\theta\subset F$, there is a vector field $E$ with these properties and
linearly independent with $\partial_{s}$ in $\theta$. Hence, any plane
containing $E$ has the same sectional curvature. Suppose that $E=\alpha
\partial_{s}+U$ with $U_{(s,z)}\in T_{z}\mathbb{S}^{2}$ for all $(s,z)\in
\theta$ and consider the planes $\Pi_{0}=span\left(  E,\partial_{s}\right)  $
and $\Pi_{1}=\left(  E,V\right)  $ where $V_{(s,z)}\in T_{z}\mathbb{S}^{2}$
with $V\perp U$. A straightforward computation shows that
\begin{align*}
K(\Pi_{0}) &  =-\frac{\mu^{\prime\prime}}{\mu},\\
K\left(  \Pi_{1}\right)   &  =-\alpha^{2}\frac{\mu^{\prime\prime}}{\mu}%
+\frac{\left(  1-\alpha^{2}\right)  \left(  1-\mu^{\prime2}\right)  }{\mu^{2}%
}.
\end{align*}

Therefore, since $K(\Pi_{0})=K(\Pi_{1})$, it holds $\mu^{\prime\prime}%
=\frac{\mu^{\prime2}-1}{\mu}$. The solutions to this differential equation are
$\mu(s)=\frac{1}{k}\sinh\left(  ks+s_{0}\right)  $, $\mu(s)=\frac{1}{k}%
\sin\left(  ks+s_{0}\right)  $ or $\mu(s)=\pm s+s_{0}$ but for these warping
functions the sectional curvature of $ds^{2}+\mu(s)^{2}g_{0}$ is $-k^{2}$,
$k^{2}$ or $0$ respectively.
\end{proof}

\begin{theorem}
\label{Teor2}If $h\in C^{\infty}(I)$ is a positive function such that
$h^{\prime\prime\prime}\not \equiv 0$ in any open subset of $I$, then the
spacetime given by (\ref{familiaestaticos}) has exactly two totally umbilic
null hypersurface through each point.
\end{theorem}

\begin{proof}
Call $\left(  F,g_{F}\right)  =\left(  I\times\mathbb{S}^{2},\frac{1}%
{h(r)}dr^{2}+r^{2}g_{0}\right)  $. If we take a function $\varphi$ such that
$\varphi^{\prime}(s)=h\left(  \varphi(s)\right)  $ and we make the coordinate
change $r=\varphi(s)$, then $\frac{1}{h}g_{F}$ is written as $ds^{2}%
+\mu(s)^{2}g_{0}$, where $\mu(s)=\frac{\varphi(s)}{\sqrt{h\left(
\varphi(s)\right)  }}$. Since $\left(  \frac{\mu^{\prime\prime}}{\mu}\right)
^{\prime}=-\frac{h^{2}h^{\prime\prime\prime}}{2}$, there is not any open
neighborhood in $\left(  F,\frac{1}{h}g_{F}\right)  $ with constant curvature.
Applying the above lemma, this warped decomposition is unique. Therefore,
using Theorem \ref{estatico} and Remark \ref{hipersuperficiedualestatico}, for
each point there are exactly two totally umbilic null hypersurfaces.
\end{proof}

\begin{corollary}
In a De Sitter-Schwarzschild with $m\neq0$ and in a Reissner-Nordstr\"{o}m
spacetime there are exactly two totally umbilic non-totally geodesic null
hypersurface through each point.
\end{corollary}

\begin{proof}
From Equations (\ref{curvaturamediaestatico1}) and
(\ref{curvaturamediaestatico2}), the null mean curvatures are $H^{\ast}%
=2\frac{d}{ds}\ln\varphi$ and $\widetilde{H}^{\ast}=2\frac{d}{ds}\ln\left(
\frac{h}{\varphi}\right)  $, which are not identically zero.
\end{proof}

\begin{remark}
If we consider the Schwarzschild exterior embedded in the Kruskal spacetime
$Q\times_{r}\mathbb{S}^{2}$, the totally umbilic null hypersurfaces claimed in
the above corollary are given by
\[
\{(u,v,x)\in Q\times\mathbb{S}^{2}:u=u_{0}\}
\]
and%
\[
\{(u,v,x)\in Q\times\mathbb{S}^{2}:v=v_{0}\}.
\]

\end{remark}

\noindent
\textbf{Manuel Guti\'errez.}\\
m\_gutierrez@uma.es\\ Dep. \'{A}lgebra, Geometr\'{\i}a y Topolog\'{\i}a.\\Universidad de M\'{a}laga.M\'{a}laga. Spain

\noindent
\textbf{Benjam\'in Olea.}\\
 benji@uma.es\\ Dep. Matem\'{a}tica Aplicada.\\Universidad de M\'{a}laga, M\'{a}laga, Spain

\end{document}